\documentclass[12pt]{article}
\usepackage[utf8]{inputenc}

\usepackage{amsmath, amsfonts, amssymb, amsthm, framed}
\usepackage[top=1in, bottom=1in]{geometry}
\newtheorem{theorem}{Theorem}[section]
\newtheorem{lemma}[theorem]{Lemma}
\newtheorem{proposition}[theorem]{Proposition}
\newtheorem{corollary}[theorem]{Corollary}

\theoremstyle{definition}

\newtheorem{definition}[theorem]{Definition}
\newtheorem{remark}[theorem]{Remark}

\usepackage{enumerate}
\usepackage{expl3}
\usepackage{xparse}

\usepackage{graphicx}
\usepackage{wrapfig}
\usepackage{float}
\usepackage{subcaption}

\usepackage{graphicx}
\usepackage[colorinlistoftodos]{todonotes}
\usepackage[colorlinks=true, allcolors=blue]{hyperref}

\usepackage{enumerate}
\usepackage{expl3}
\usepackage{xparse}

\newcommand{\N}{\mathbb{N}}

\newcommand{\Z}{\mathbb{Z}}

\usepackage{color}

\newcommand{\ignore}[1]{}

\ExplSyntaxOn
\NewDocumentCommand{\cycle}{ O{\;} m }
 {
  (
  \alec_cycle:nn { #1 } { #2 }
  )
 }

\seq_new:N \l_alec_cycle_seq
\cs_new_protected:Npn \alec_cycle:nn #1 #2
 {
  \seq_set_split:Nnn \l_alec_cycle_seq { , } { #2 }
  \seq_use:Nn \l_alec_cycle_seq { #1 }
 }
\ExplSyntaxOff

\title{\vspace*{-0.1in}Palindromes in finite groups and the Explorer-Director game}
\date{March 31, 2019}

\author{Dagur T\'omas \'Asgeirsson\thanks{Faculty of Physical Sciences, University of Iceland, Reykjav\'ik \qquad \texttt{dta1@hi.is}} \and Pat Devlin\footnote{Dept.\ of Mathematics, Yale University, New Haven CT \qquad \texttt{patrick.devlin@yale.edu}}}

\begin{document}
\renewcommand{\thefootnote}{\fnsymbol{footnote}}
\footnotetext{AMS 2010 subject classification: 05E15, 20F10, 20F18, 20F69}
\footnotetext{Key words and phrases:  Explorer-Director game, Magnus-Derek game, twisted subgroups, palindromes, combinatorial group theory, nilpotent groups}
\maketitle

\begin{abstract}
In this paper, we use the notion of \textit{twisted subgroups} (i.e., subsets of group elements closed under the binary operation $(a,b) \mapsto aba$) to provide the first structural characterization of optimal play in the \textit{Explorer-Director game}, introduced as the \textit{Magnus-Derek game} by Nedev and Muthukrishnan and generalized to finite groups by Gerbner.  In particular, we reduce the game to the problem of finding the largest proper twisted subgroup, and as a corollary we resolve the Explorer-Director game completely for all nilpotent groups.
\end{abstract}
\renewcommand*{\thefootnote}{\alph{footnote}}
\section{Introduction}\label{section introduction}
In \cite{magnus-derek-original}, Nedev and Muthukrishnan introduced the following game between two players, whom we call Explorer and Director.\footnote{This was originally called the \textit{Magnus-Derek game} (Magnus determining the magnitude and Derek the direction).  However, we choose to say \textit{Explorer-Director game}, which we feel is more descriptive.}  The game is played in rounds moving a token around a circular table with $n$ labeled positions.  Each round, Explorer names a distance by which the token is to be moved, and Director then determines a direction and moves the token the specified amount.  Explorer's goal is to minimize the number of unvisited positions while Director's is to maximize this number.  
It was shown in \cite{magnus-derek-original} that the number of positions visited on an $n$ element table assuming optimal play is 
\[
f^*(n) = \begin{cases}n, \qquad &\text{if $n$ is a power of $2$,}\\ n(1 - 1/p), \qquad &\text{where $p$ is the smallest odd prime factor dividing $n$}.\\ \end{cases}.
\]

Motivated by algorithmic aspects of the problem, this was extended by Hurkens, Pendavingh, and Woeginger \cite{MR2466788} addressing the question of how efficiently Explorer can reach $f^*(n)$ positions.  This direction was further developed by Nedev \cite{MR2725450, MR3187798} as well as by Chen, Lin, Shieh, and Tsai \cite{MR2778466}, who also introduced some variants of the game.

In 2013, Gerbner \cite{magnus-derek-groups} posed an algebraic generalization where the game positions are elements of some finite group, $G$.  Each round, Explorer picks an element $g \in G$ and Director decides whether to right-multiply the current position by $g$ or $g^{-1}$.  In this language, the original game is the special case that $G$ is the cyclic group $\Z/n\Z$.

Let $f(G)$ denote the number of group elements that are visited assuming optimal play in a group $G$.  For abelian groups, Gerbner proved $f(G) = f^* (|G|)$, directly generalizing the result of \cite{magnus-derek-original}.  Moreover, for $2$-groups and for groups generated by involutions, he showed $f(G) = |G|$.  For general groups, nothing else has been established.

In this paper, we make progress in understanding $f(G)$ by relating the game to the following natural notion in combinatorial group theory \cite{nearSubgroups}.

\ignore{
\begin{definition}\label{def closed under roots}
Let $G$ be a group and $R \subseteq G$. We say that $R$ is \emph{closed under square roots} iff $ax, ax^{-1} \in R$ implies $a \in R$.
\end{definition}
}

\begin{definition}\label{def palindromic}
Let $G$ be a group and $P \subseteq G$. We say $P$ is a \emph{twisted subgroup of $G$} iff $P$ contains the identity and $aba\in P$ whenever $a,b\in P$.
\end{definition}

Twisted subgroups date back to Glauberman \cite{loops1, loops2}, and they have since been studied in connection with certain algebraic optimization problems (see \cite{feder, nearSubgroups} as important examples and \cite{twistedSurvey} for a survey).

For $|G|$ odd, this definition enables us to provide the following \textit{structural characterization} of which subsets can arise as the set of unvisited positions assuming optimal play.  In doing so, we essentially characterize the best strategies for each player, and we are able to reduce the study of $f(G)$ to a question regarding twisted subgroups.
\begin{theorem}\label{theorem magnus derek palindrome}
Let $G$ be a group of odd order.  If $U$ is attainable as a set of unvisited positions under optimum play, then $U$ is a coset of a proper twisted subgroup.  Moreover, if $C$ is a coset of a twisted subgroup and if the token does not start in $C$, then there is a strategy for Director that prevents the token from visiting any of the elements of $C$.
\end{theorem}

\begin{corollary}\label{corollary from main theorem odd}
Let $G$ be a group of odd order.  Then $f(G) = |G| - \max_{P} |P|$, where the maximum is taken over all proper twisted subgroups of $G$.
\end{corollary}
We also obtain a similar but more nuanced characterization for all $G$ (Lemma \ref{lemma maximal sets in modified game}), but instead of appealing to that corresponding result, it is more convenient for us to reduce the problem of finding $f(G)$ to the case of $|G|$ odd via the following theorem.

\begin{theorem}\label{theorem reducing to odd}
Let $G$ be a finite group and $\Gamma$ the subgroup of $G$ generated by elements whose orders are powers of $2$.  Then $\Gamma \lhd G$, $|G / \Gamma|$ is odd, and $f(G) = f(G / \Gamma) |\Gamma|$.
\end{theorem}

This enables us to focus on twisted subgroups in groups of odd order (the even order case being more complicated for a variety of reasons discussed in the conclusion).  Although every subgroup is in fact a twisted subgroup, unfortunately there are twisted subgroups that are not closed under the group operation.  Nonetheless, we have the following result of Glauberman.

\begin{theorem}\label{theorem palindromic divides}\cite{loops1}
If $P \subseteq G$ is a twisted subgroup and $|G|$ is odd, then $|P|$ divides $|G|$.
\end{theorem}

This shows that for odd order groups, $f^* (|G|) \leq f(G)$, compared to the upper bound $f(G) \leq |G| - \max_{H} |H|$, where the maximum is taken over all proper subgroups of $G$.  For odd groups having a subgroup of minimum prime index, these bounds coincide.  In particular, we obtain the following clean statement for nilpotent groups.
\begin{theorem}\label{theorem nilpotent any not nec odd}
If $G$ is nilpotent (with $|G|$ not necesssarily odd), $f(G) = f^* (|G|)$.
\end{theorem}

\subsection*{Structure of paper}
We begin in Section \ref{section even}, where we relate the game on $G$ to the game on its quotient groups, which leads to a proof of Theorem \ref{theorem reducing to odd} (thus reducing the problem to the odd order case).  In Section \ref{section odd}, we then introduce an equivalent auxilliary game, which turns out to be easier to analyze.  For odd order groups, we make the connection to twisted subgroups, and we prove Theorem \ref{theorem magnus derek palindrome} and Corollary \ref{corollary from main theorem odd}.  For completeness, we provide a self-contained proof of Theorem \ref{theorem palindromic divides} in Section \ref{section palindromes}.  We conclude in section \ref{section conclusion} with a discussion of open problems and some cautionary examples illustrating the falsehood of several natural conjectures.

\section{Reducing to the case $|G|$ odd}\label{section even}
In this section, we reduce the problem for general groups to the odd order case by proving Theorem \ref{theorem reducing to odd}. This is via the following two lemmas.  Lemma \ref{lemma coset upper bound} bounds $f(G)$ in terms of its quotient groups, and its proof serves as a good warm up to the problem.

\begin{lemma}\label{lemma coset upper bound}
If $K \lhd G$, then $f(K) f(G/K) \leq f(G) \leq |K| f(G/K).$
\end{lemma}
\begin{proof}
To find a lower bound on $f(G)$, consider the following strategy that Explorer could use for playing in $G$:
\begin{itemize}
\item[(a)] Each time the token arrives in a new left coset $gK$, Explorer chooses only elements of $K$ (thereby staying within that coset) until she has moved the token to as many new positions within $gK$ as possible.
\item[(b)] By playing as if in $G/K$, Explorer then moves the token to a new coset if possible.
\end{itemize}
If Explorer follows this strategy, the token will visit at least $f(K)$ elements within each coset, and it will visit at least $f(G/K)$ cosets.

As for the upper bound, when playing in $G$, Director could act as if in $G/K$ with the singular goal of responding such that the token reaches at most $f(G/K)$ cosets.
\end{proof}

Next we extend a result of Gerbner \cite{magnus-derek-groups}, who proved a special case of the following.
\begin{lemma}\label{lemma equality for powers of two}
If $\Gamma$ is generated by elements whose orders are powers of $2$, $f(\Gamma) = |\Gamma|$.
\end{lemma}
\begin{proof}
Suppose the token is currently at $x$ and $t$ is an element of order $2^k$.  We will show that Explorer has a strategy to move the token from $x$ to $xt$, which (by our assumption on $\Gamma$) will show Explorer can ensure the token visits every position.

For Explorer to move the token from $x$ to $xt$, she performs the following algorithm:
\begin{itemize}
\item Explorer chooses $t^1$, $t^2$, $t^4$, $t^8$, \ldots , $t^{2^{k-1}}$ in order until the token is at $xt$.
\end{itemize}
To prove this strategy works, we need only show that one of Director's responses necessarily moves the token to $xt$.  If the token does not reach $xt$ before Explorer chooses $t^{2^{k-1}}$, then for each $t^{2^i}$, Director must have responded $x t^{1-2^{i}} \mapsto xt^{1-2^{i+1}}$ (otherwise, Director's first deviation from this strategy would put the token at $xt$).  The token would then be at $x t^{1-2^{k-1}}$ when Explorer declares $t^{2^{k-1}}$, and both of Director's responses to this would move the token to $xt = xt^{1-2^{k}}$ since $t$ has order $2^k$.
\end{proof}

With these two lemmas, Theorem \ref{theorem reducing to odd} follows immediately, which reduces the problem to the case $|G|$ odd.
\ignore{\begin{lemma}\label{lemma powers of 2 is normal}
Let $G$ be a finite group and $\Gamma \subseteq G$ be the subgroup of $G$ generated by all the elements whose orders are powers of $2$.  Then $\Gamma \lhd G$ and $|G / \Gamma|$ is odd.
\end{lemma}
\begin{proof}
First note that $|G| / |\Gamma|$ is odd since $\Gamma$ contains a (in fact every) Sylow $2$-group of $G$.  To see that $\Gamma \lhd G$, let $x \in \Gamma$ and write $x = t_1 t_2 \cdots t_k$, where each $t_i$ has order a power of $2$.  Then conjugating we see
\[
g x g^{-1} = g \left(\prod_{i\leq k} t_i \right) g^{-1} =  \prod_{i\leq k} (g t_i g^{-1}),
\]
and since conjugation does not change order of an element, each $g t_i g^{-1}$ has order a power of $2$.
\end{proof}

The previous three lemmas show that for any group $G$, with $\Gamma$ defined like in Lemma \ref{lemma powers of 2 is normal}, we have
\[
    f(G) = |\Gamma|f(G/\Gamma).
\]
Therefore, it suffices to find $f(G/\Gamma)$. Further, the order of the group $G/\Gamma$ is always odd. Hence we have reduced the problem to the odd order case.
}

\section{An equivalent game and $|G|$ odd}\label{section odd}
To better understand the original game, we define the \textit{open Explorer-Director game} as follows.  Director first picks a set $U \subseteq G$ and \textit{tells Explorer what that set is}.  Director's goal is to pick as large a set as possible and to always direct the token so as to stay out of $U$.  In this version, Explorer's only goal is to move the token into $U$.  We then define $\tilde{f}(G) = |G| - \max_{U} |U|$, where the maximum is taken over all sets $U$ for which Director can win this modified game.

Conveniently, the open Explorer-Director game is equivalent to the original.

\begin{proposition}\label{proposition games equivalent}
For any finite group $G$, $\tilde{f}(G) = f(G)$.
\end{proposition}
\begin{proof}
In the original game, Director can pick any set that would win the open game (without telling Explorer) and play as if playing the open game. Thus $f(G) \leq \tilde{f}(G)$.

Now we consider the original game from Explorer's point of view. Suppose $U$ is the set of elements the token hasn't visited.  Note that if $|U| > |G| - \tilde{f}(G)$, then Explorer can pretend the set $U$ was chosen in the open game and play so as to make the token reach some element of $U$ (Explorer can do that since otherwise $U$ would be a larger set for which Director can win the open game).  Thus, Explorer can always ensure $|U| \leq |G|-\tilde{f}(G)$, which implies $f(G) \geq \tilde{f}(G)$.
\end{proof} 

It turns out that the open game is much nicer to analyze.  In fact, the optimal strategies for each player can easily be described in terms of the following.

\begin{definition}\label{def between}
Let $G$ be a group. We say that an element $b \in G$ is \textit{between} elements $a,c \in G$ iff there exists $g \in G$ such that $a = bg$ and $c = bg^{-1}$ [i.e., iff $a = bc^{-1} b$].  Moreover, we say a subset $B \subseteq G$, is \textit{closed under betweenness} iff for all $a,c \in B$, if $b$ is between $a$ and $c$, then $b \in B$.
\end{definition}

Note that this definition is valid for all groups, and in general given $a$ and $c$, there may be multiple elements between them.  For instance, if a group is generated by involutions, then a nonempty subset is closed under betweenness iff it is the entire group (since in that case, for all $a \in G$ and any involution $t$, $at$ is between $a$ and $a$).  We will show that these notions are particularly well-behaved in groups of odd order, but in general they characterize the sets for which Director can win the open game.

\begin{lemma}\label{lemma maximal sets in modified game}
For any group $G$, if $U$ is a maximal set for which Director can win the open game, then Explorer can reach every element outside of $U$, and $U$ must be closed under betweenness.  Moreover, if $B \subseteq G$ is closed under betweenness and the token does not start in $B$, then Director can win the open game by declaring $B$.
\end{lemma}
\begin{proof}
First, if there were an element $y \notin U$ that Explorer could not reach, then Director could win the open game with $U \cup \{y\}$ (violating maximality).  To see that $U$ is closed under betweenness, suppose $b \notin U$.  We know Explorer can reach $b$, and if she chooses $g$, Director must have some response that keeps the token from $U$.  But the token reaches either $bg$ or $b g^{-1}$, so these cannot both be in $U$, and hence $b$ cannot be between any two elements of $U$.

As for the second part of the claim, suppose $B$ is closed under betweenness and the token is at a point $x \notin B$.  Then Director can respond to any $g$ so as to keep the token out of $B$ otherwise, both $xg$ and $xg^{-1}$ would be in $B$, which (because $B$ is closed under betweenness) would require $x \in B$ as well.
\end{proof}

This lemma provides an analog of Theorem \ref{theorem magnus derek palindrome} (and Corollary \ref{corollary from main theorem odd}) valid for all groups $G$.  However, in the case of $|G|$ odd the situation is much better behaved.  In fact, we establish Theorem \ref{theorem magnus derek palindrome} simply by combining Lemma \ref{lemma maximal sets in modified game} with the following.

\begin{proposition}\label{proposition closed under betweenness iff palindromic}
Let $G$ be a group of odd order.  A set $B \subseteq G$ is closed under betweenness iff there exists a twisted subgroup $P \subseteq G$ and some $g \in G$ such that $B = gP$.
\end{proposition}
\begin{proof}
First, note that in a group of odd order, the map $g \mapsto g^2$ is a bijection (as iterating it sufficiently often results in the identity because $2$ and $|G|$ are relatively prime).  Thus, for every $x \in G$, there is exactly one $y$ (called \textit{the square root} of $x$) such that $x = y^2$.  Similarly, for all $a, c \in G$, there is exactly one element, which we denote $b = b(a,c)$, that is between $a$ and $c$.  This is because we have
\[
a = bg \text{ and } c = bg^{-1} \Longleftrightarrow b = cg \text{ and } g^2 = c^{-1}a,
\]
and $c^{-1}a$ has exactly one square root (since $|G|$ is odd).  With this, we are now able to prove a set is closed under betweenness iff it is a coset of a twisted subgroup.

\paragraph*{Forward implication:} Suppose $U$ is closed under betweenness.  If $c$ and $ct$ are in $U$, we claim that $c t^{-1} \in U$.  To prove this, for each $x \in G$, let $S(x)$ denote the square root of $x$, and let $S^k$ denote $S$ composed with itself $k$ times.  Because $c S(x) = b(c, cx)$, we have $c S^{k} (t) \in U$ for all $k \geq 0$.  But since $S$ is the inverse map of $x \mapsto x^2$, finiteness gives us $U \supseteq \{c S^{k} (t) \ : \ 0 \leq k \in \mathbb{Z}\} = \{c t^{2^k} \ : \ 0 \leq k \in \mathbb{Z}\}$.  After seeing that $c t^{2^k} \in U$ for all $k$, we get that $c t^j \in U$ for all integers $j \geq 1$ by noting that $\{j \in \mathbb{Z} \ : \ c t^j \in U\}$ is closed under taking averages of even numbers.

\ignore{
Suppose $m$ is the order of $t$ and that $2^r$ is the largest power of 2 which does not exceed $m$. Let $s \leq r$ be an integer and observe that $ct^k \in U$ for all integers $k$ between $2^{s-1}$ and $2^s$ --- simply keep taking the \lq\lq between\rq\rq\   elements. Hence $\{ct^k\}_{k = 0}^{2^r} \subseteq U$, i.e.~more than half of all the elements $ct^k$ are contained in $U$. For each integer $q$ between $1$ and $\frac{m+1}{2}$, we have 
\[
b(ct^{q-1}, ct^q) = ct^{q-1}S(x) =: ct^k
\] 
where $k > \frac{m+1}{2}$. Further, for the $\frac{m+1}{2}$ different choices of $q$ we get $\frac{m+1}{2}$ different values of $k$. We are done, since $\frac{m+1}{2} < 2^r$ and hence $ct^k \in U$ for all these choices of $q$.
}

We now show how the rest of the proof follows from this.  Fix any $a \in U$, and let $x, y \in a^{-1} U$ be arbitrary.  Because $a, ay \in U$, our claim implies $ax(x^{-1} y^{-1}) = a y^{-1} \in U$.  Since $ax, ax (x^{-1} y^{-1}) \in U$, our claim again implies $axyx = ax (x^{-1} y^{-1})^{-1} \in U$.  This therefore shows that $a^{-1} U$ is closed under palindromes, as desired.

\paragraph*{Reverse implication:} Suppose $P \subseteq G$ is a twisted subgroup of $G$, and let $g \in G$ and $x, y \in P$ be arbitary.  Suppose $y$ has order $n$, and that the order of $y^{-1}x$ is $2m-1$. Then $(y^{-1}x)^m$ is the square root of $y^{-1}x = (gy)^{-1} gx$, and thus we have
\[
b(gx, gy) = gy (y^{-1}x)^m = g x (y^{-1} x)^{m-1} = g x (y^{n-1} x)^{m-1}.
\]
Moreover, since $P$ is closed under forming palindromes, we know $g x (y^{n-1} x)^{m-1} \in gP$ since the word $x (y^{n-1} x) ^{m-1} = x y^{n-1} x y^{n-1} \cdots y^{n-1} x$ is a palindrome using elements of $P$.  Thus $gP$ is closed under betweenness.
\end{proof}

\section{Twisted subgroups in groups of odd order}\label{section palindromes}
We now turn our attention to twisted subgroups $P \subseteq G$ where $|G|$ is odd.  In particular, we prove Theorem \ref{theorem palindromic divides}, that $|P|$ divides $|G|$.  Though originally appearing in \cite{loops1}, we offer a self-contained proof here for completeness.

\begin{proof}[Proof of Theorem \ref{theorem palindromic divides}]
Let $G$ be a group of odd order and $P \subseteq G$ be a twisted subgroup of $G$.  We will prove $|P|$ divides $|G|$ by induction on $|G|$.  The case $|G| = 1$ is trivial, so let $|G| > 1$ and assume the claim holds for all groups of order $2k + 1 < |G|$.

We may assume that $P$ generates $G$ since otherwise we can apply our argument to the subgroup generated by $P$.  Let $\mathcal{W}$ denote the set of words in the alphabet $P$.  For $w = (w_1, w_2, \ldots, w_l) \in \mathcal{W}$, let $\overline{w} = (w_l , w_{l-1} , \ldots, w_1)$ denote the reversal of $w$, and let $|w| = w_1 w_2 \cdots w_l \in G$ denote the evaluation of the word $w$ viewed as a product of elements of $G$.  Clearly, for any words $u,v \in \mathcal{W}$, we have $\overline{u \cdot v}= \overline{v} \cdot \overline{w}$ and $|u\cdot v| = |u| |v|$.

In studying the notion of \textit{palindromic width of a group}, Fink and Thom \cite{fink} considered the set $H = \{ |\overline{w}| \in G \ : \ w \in \mathcal{W}, \ |w| = 1_{G}\}$---where $1_{G} \in G$ denotes the identity element of $G$---and they showed that $H$ is a normal subgroup of $G$.  We also know that $H \subseteq P$ since every $g \in H$ is of the form $g = |\overline{w}| = |\overline{w}| 1_{G} = |\overline{w}| |w| = |\overline{w} w| \in P$.

Because $H \subseteq P$ and $P$ is a twisted subgroup, we claim $HP \subseteq P$.  To see this, let $h \in H$ and $p \in P$ be arbitrary.  Because $|G|$ is odd, we know we $p = q^2$ for some $q \in G$ and moreover $q \in P$ since $q$ is a power of $p$.  Since $H \lhd G$, we know $q^{-1} h q \in H \subseteq P$, which implies $hp = h q^2= q (q^{-1} h q) q \in P$.  Thus, $HP \subseteq P$, and so $HP = P$.

Let $\pi : G \to G/H$ be the canonical projection map.  The image of $P$ under this map is a twisted subgroup (of $G/H$) since $\pi(a)\pi(b)\pi(a) = \pi(aba) \in \pi(P)$.  Thus, if $H \neq \{1_G \}$, we could apply induction to say that $|\pi(P)|$ divides $|G/H|$, and since $P = HP$, we would have $|P| = |H| \cdot |\pi(P)|$ divides $|G|$.

We may therefore assume that $H = \{1_G\}$.  For two words $u,v \in \mathcal{W}$, we have
\[
|u| = |v| \quad \Longleftrightarrow \quad  u_1 u_2 \cdots u_n = v_1 v_2 \cdots v_m \quad \Longleftrightarrow \quad u_1 u_2 \cdots u_n v_m ^{-1} v_{m-1} ^{-1} \cdots v_{1} ^{-1} = 1_G.
\]
Because $H = \{1_G\}$, we know that $|w| = 1_G$ iff $|\overline{w}| = 1_G$.  Thus, continuing the above
\[
|u| \cdot v_m ^{-1} v_{m-1} ^{-1} \cdots v_{1} ^{-1} = 1_G \quad \Longleftrightarrow \quad v_1 ^{-1} v_2 ^{-1} \cdots v_{m} ^{-1} \cdot |\overline{u}| = 1_G  \quad \Longleftrightarrow \quad |\overline{u}| = |\overline{v}|.
\]
Therefore, since $|u| = |v|$ iff $|\overline{u}| = |\overline{v}|$, this gives a well-defined map $\psi : G \to G$ via $\psi(|w|) = |\overline{w}| ^{-1}$.  Moreover, it is easy to see that $\psi$ is an order 2 automorphism of $G$.

Let $N = \{g \in G \ : \ \psi(g) = g\} \lhd G$.  Note that $\psi(p) = p^{-1}$ for all $p \in P$.  Thus for $m, n \in N$ and $p,q \in P$, if $mp = nq$ then $n=m$ and $p=q$ (by applying $\psi$ to both sides, rearranging, and using the injectivity of the map $x \mapsto x^2$ [valid since $|G|$ is odd]).  On the other hand, for arbitrary $g \in G$, note that $q = (\psi(g))^{-1} g$ is an element of $P$, so we can write $q = p^2$ for some $p \in P$.  Thus $g = \psi(g) p^2$, and right-multiplying both sides of this by $p^{-1}$ and applying $\psi$, we see that $\psi(g) p \in N$, and so $g \in NP$.

Combining these two observations, we have that every element of $G$ is uniquely expressible as $np$ for $n \in N$ and $p \in P$, implying $|P| = |G| / |N|$, as desired.
\end{proof}

\section{Conclusion and further research}\label{section conclusion}
From Theorems \ref{theorem magnus derek palindrome} and \ref{theorem palindromic divides}, we see that for any group $G$ of odd order
\[
f^* (|G|) \leq f(G) \leq |G|-\max_{H \lneqq G} |H|,
\]
where the maximum is taken over all proper subgroups of $G$, and the most natural open question would be to improve these bounds. It is tempting to conjecture that for $|G|$ odd, $f(G)$ is actually equal to one of these two bounds, but it's not clear to us which of these bounds is more likely to be the truth.

\ignore{
Now suppose $G$ is any nilpotent group (with $|G|$ not necessarily odd). Then $(*)$ holds with $G$ replaced by $G/\Gamma$, which is also a nilpotent group and hence has a subgroup of any given order dividing its order (see \cite[Theorem 3, p.~191]{dummitfoote}). Thus we have equality everywhere in $(*)$. Since $G$ is nilpotent, $\Gamma$ is its unique Sylow 2-subgroup, and thus the smallest odd prime dividing $|G|$ and $|G/\Gamma|$ is the same. We conclude that
\[
f(G) = |\Gamma|f(G/\Gamma) = |\Gamma|f^*(|G/\Gamma|) = f^*(|G|).
\]
Now we have proven Theorem \ref{theorem nilpotent any not nec odd}.

In fact, we can conclude that for any group $G$ of odd order with a subgroup of index $p$ where $p$ is the smallest prime dividing $|G|$,
\[
f(G) = f^*(|G|).
\]}

Of course, for $|G|$ odd, determining $f(G)$ amounts to understanding for which values $k$ a twisted subgroup of size $k$ exists.  Let $L(G) = \{|P| \ : \ P \subsetneqq G, \ \text{$P$ a twisted subgroup}\}$.  We know $k$ divides $|G|$ for all $k \in L(G)$, and that $f(G) = |G| - \max (L(G))$.  While it is important to note that twisted subgroups need not be subgroups (e.g., there are small counterexamples in non-abelian groups of order 27 and 75), it could perhaps be the case that $L(G) = \{ |H| \ : \ H \lneqq G\}$ (and thus, our upper bound on $f(G)$ would be equality).  Perhaps an equally daring conjecture would be that $L(G) = \{ d \ : \ \text{$d$ divides $|G|$}\}$, in which case we would have $f(G) = f^*(|G|)$.  Admittedly, we would not be particularly surprised if both of these conjectures were false, but we were unable to disprove either.

Though not directly related to understanding the game, the case of $|G|$ even is known to be particularly more nuanced (see for instance \cite{twistedSurvey}).  For instance, in the abelian group $(\mathbb{Z}/2\mathbb{Z})^n$, \textit{every} subset containing the identity element is a twisted subgroup.  As a more mild example, in the dihedral group on $2n$ elements, the set $\{f\} \cup \{r^k \ : \ k \in \mathbb{Z}\}$ consisting of $n+1$ elements is also a twisted subgroup.  Thus, while the study of twisted subgroups for groups of even order is interesting in its own right, it certainly seems to be much more difficult.

\bibliographystyle{plain}
\bibliography{references}

\end{document}